\documentclass[11pt]{amsart}

\usepackage{amsmath}
\usepackage{amscd}
\usepackage{amssymb}
\usepackage{amsthm}
\usepackage{latexsym}
\usepackage{mathabx}
\usepackage{mathtools}
\usepackage{hyperref}
\usepackage{stmaryrd}
\usepackage{tikz-cd}
\usepackage{enumitem}
\usepackage{accents}
\usepackage{dsfont}

\usepackage[arrow,curve,matrix,tips,2cell]{xy}
  \SelectTips{eu}{10} \UseTips
  \UseAllTwocells

\usepackage{calrsfs}
\usepackage[T1]{fontenc} 
\usepackage{textcomp}
\usepackage{times}
 \usepackage[scaled=0.92]{helvet}

\usepackage{geometry}
\newtheorem{theorem}{Theorem}[section]
\newtheorem{lemma}[theorem]{Lemma}
\newtheorem{corollary}[theorem]{Corollary}

\newtheorem{theoremA}{Theorem}

\theoremstyle{definition}
\newtheorem{remark}[theorem]{Remark}
\newtheorem{definition}[theorem]{Definition}
\theoremstyle{plain}

\newtheorem*{corollarystar}{Corollary}
\newtheorem*{definition*}{Definition}

\newcommand{\cO}{\mathcal O}
\newcommand{\cP}{\mathcal P}

\newcommand{\cS}{\mathcal S}

\def\Cz{\mathbb{C}}

\def\Nz{\mathbb{N}}
\def\Qz{\mathbb{Q}}

\def\Zz{\mathbb{Z}}

\newcommand{\mfp}{\mathfrak p}
\newcommand{\mfq}{\mathfrak q}
\newcommand{\mfP}{\mathfrak P}

\newcommand{\Gal}{\mathrm{Gal}}

\newcommand{\pk}{\mathcal{P}_K}
\newcommand{\pkx}{\mathcal{P}_{K'}}
\newcommand{\pks}{\mathcal{P}^{(1)}_K}
\newcommand{\pksx}{\mathcal{P}^{(1)}_{K'}}

\newcommand{\Hqab}{\widecheck{G}_{\Qz}}
\newcommand{\Hkab}{\widecheck{G}_{K}}
\newcommand{\Hkabx}{\widecheck{G}_{K'}}
\newcommand{\Hktauab}{\widecheck{G}_{K_\tau}}
\newcommand{\Hnab}{\widecheck{G}_{N}}
\newcommand{\Frob}[1]{\textup{Frob}_{#1}}

\newcommand{\HomL}{\textup{Hom}_{L\textup{-series}}(\Hkab, \Hkabx)}
\newcommand{\HomLtau}{\textup{Hom}_{L\textup{-series}}(\Hktauab, \Hkabx)}

\newcommand{\HomLl}{\textup{Hom}_{L\textup{-series}}(\Hkab[l], \Hkabx[l])}

\newcommand{\HomPld}{\textup{Hom}^\delta_{\mathcal{P}}(\Hkab[l], \Hkabx[l])}
\newcommand{\HomPNl}{\textup{Hom}_{\mathcal{P}}(\Hkab[l], \Hnab[l])}

\newcommand{\HomQPld}{\textup{Hom}_{\mathcal{P}}^\delta(\Hqab[l], \Hkab[l])}
\newcommand{\HomQL}{\textup{Hom}_{L\textup{-series}}(\Hqab[l], \Hkab[l])}

\newcommand{\Homfield}{\textup{Hom}(K, K')}
\newcommand{\HomfieldN}{\textup{Hom}(K, N)}

\newcommand{\HomfieldQ}{\textup{Hom}(\Qz, K)}

\newcommand{\DistTo}{\xrightarrow{
   \,\smash{\raisebox{-0.45ex}{\ensuremath{\scriptstyle\sim}}}\,}}

\newcommand\restr[2]{{
  \left.\kern-\nulldelimiterspace 
  #1 
  \vphantom{\big|} 
  \right|_{#2} 
  }}

\newcommand{\pow}{\mathbb{P}} 
\newcommand{\im}{\textup{im}}

\begin{document}

\date{\today\ (version 1.0)} 
\title{$L$-series and homomorphisms of number fields}
\author[H.~Smit]{Harry Smit}
\address{Max Planck Institute for Mathematics, Vivatsgasse 7, 53111 Bonn, Germany}
\email{smit@mpim-bonn.mpg.de}
\subjclass[2010]{11R37, 11R42}
\keywords{\normalfont Class field theory, $L$-series, arithmetic equivalence}

\begin{abstract}
While the zeta function does not determine a number field uniquely, the $L$-series of a well-chosen Dirichlet character does. Moreover, isomorphisms between two number fields are in natural bijection with $L$-series preserving isomorphisms of $l$-torsion subgroups of the Dirichlet character groups. We extend this by showing that homomorphisms between number fields are in natural bijection with group homomorphisms between $l$-torsion subgroups of the Dirichlet character groups abiding a divisibility condition on the $L$-series when $l$ is sufficiently large.
\end{abstract}

\maketitle

\section*{Introduction}
A number field is not uniquely determined up to isomorphism by its zeta function, i.e.\ there exist nonisomorphic number fields with the same number of prime ideals of every norm (\cite[p.\ 671--672]{gassmann}). Even adding the ramification degrees is insufficient, as Komatsu \cite{komatsu} showed that a number field is not uniquely determined by the the adele ring either. Successful attempts at adding information so that the underlying number field is uniquely determined have been made. These additions come in the form of describing the Galois theoretic objects associated to the primes such as Frobenius elements in the absolute Galois group or a quotient of this group. 

One such example is the Neukirch-Uchida theorem (\cite[Satz 2]{neukirch} and \cite[Ch.\ XII, \S 2]{neukirchcohom}): it states that a number field is uniquely determined by the absolute Galois group. Moreover, the natural map from the set of isomorphisms between the number fields to the set of outer isomorphisms between the absolute Galois groups is a bijection \cite[Main Thm.]{uchida2}. The proof of Neukirch and Uchida shows that the decomposition groups inside the absolute Galois groups are mapped to each other, creating a norm-preserving bijection of primes, along with some information on the Frobenius elements. Progress has been made in connecting homomorphisms of function fields with homomorphisms of the absolute Galois groups as well, e.g.\ of curves defined over sub-$p$-adic fields \cite{mochizuki}. As a result, homomorphisms of function fields over $\Qz$ correpond to homomorphisms of the absolute Galois groups subject to some conditions \cite[Thm.\ B]{mochizuki}.

A second way is to add information on the Frobenius elements through Dirichlet $L$-series. Positive results have been achieved here as well. For example, for every number field $K$, even though its zeta function may not be unique, there exists a well-chosen Dirichlet character whose $L$-series is different from any Dirichlet $L$-series of any number field not isomorphic to $K$ \cite[Thm.\ 10.1]{CdSLMS}. Moreover, the natural map from isomorphisms between number fields and $L$-series preserving isomorphisms between the $l$-torsion (for prime $l$) of the Dirichlet character groups is a bijection \cite[Thm.\ A]{GDTHJS}. In this article we extend this to include \emph{homomorphisms} of number fields. We prove these are, through the natural map, in bijection with homomorphisms between the $l$-torsion (for sufficiently large prime $l$) of the Dirichlet character groups with a \emph{divisibility condition} on the $L$-series that we explain below.

The $L$-series $L(\chi, s)$ can be written as an infinite product of reciprocals of polynomials $L_p(\chi, T)$ in $T = p^{-s}$:
\[
L(\chi, s) = \prod_{\text{prime number } p} L_p(\chi, p^{-s})^{-1}.
\]
A homomorphism of number fields $K \hookrightarrow K'$ induces a map $\psi$ from the group $\Hkab$ of Dirichlet characters of $K$ to $\Hkabx$. For a character $\chi \in \Hkab$, the $L$-series of $\chi$ and $\psi(\chi)$ are connected through the following divisibility property: for any prime number $p$ the polynomial $L_p(\chi, T)$ divides $L_p(\psi(\chi, T))$, see Lemma~\ref{lemma:iota_in_HomL}. This motivates the following definition:

\begin{definition*}
Let $K, K'$ be number fields and $\chi \in \Hkab$, $\chi' \in \Hkabx$. We say that $L(\chi, s)$ \emph{divides} $L(\chi', s)$ if for every prime number $p$, $L_p(\chi, T)$ divides $L_p(\chi', T)$ as polynomials.
\end{definition*}

Denote by $\Hkab[l]$ the $l$-torsion subgroup of $\Hkab$, and let $\HomLl$ denote the set of homomorphisms $\psi\colon\Hkab[l] \to \Hkabx[l]$ such that $L(\chi, s)$ divides $L(\psi(\chi), s)$ for every $\chi \in \Hkab[l]$. Our main result is the following.

\begin{theoremA}\label{theoremA:bijection_of_homs}
Let $K, K'$ be number fields, and denote their common Galois closure by $N$. Let $l > [N:\Qz]^2$ be a prime number. Then the natural map
\begin{equation*}
\begin{tikzcd}
\Omega \colon \Homfield \arrow{r} & \HomLl.
\end{tikzcd}
\end{equation*}
is a bijection.
\end{theoremA}

(This map is denoted by $\Omega_1$ in later sections.)

One might wonder whether Theorem~\ref{theoremA:bijection_of_homs} holds for characters of small order compared to the degrees of the number fields involved. This is not the case, as we show in Section~\ref{section:counterexamples}. Therefore Theorem~\ref{theoremA:bijection_of_homs} does not make the Main Theorem of \cite{GDTHJS} obsolete.

For the proof of Theorem~\ref{theoremA:bijection_of_homs} we create an injective map between the prime ideals of the number field; we show that for any $\psi \in \HomLl$ there exists a map $\phi$ from almost all primes of $K$ to the primes of $K'$ for which
\[
\chi(\mfp)^{f_{\phi(\mfp)}} = \psi(\chi)(\phi(\mfp)),
\]
where $f_{\phi(\mfp)}$ is the inertia degree of $\phi(\mfp)$. This is different than the proof of \cite[Thm.\ A]{GDTHJS}; we cannot expect this map $\phi$ to be norm-preserving, nor bijective, nor unique (unless $K$ and $K'$ are isomorphic). As a result, showing that this results in a homomorphism $K \to K'$ (see Theorem~\ref{theorem:HomP_to_Homfield}) is done differently than the construction of the isomorphism $K \to K'$ in the proof of \cite[Thm.\ A]{GDTHJS}. It relies heavily on the assumption that $l$ is a large enough prime.

It is unknown to the author whether there exists a bijection between $\HomL$ and $\Homfield$ or not, but Theorem~\ref{theoremA:bijection_of_homs} has the following as an immediate corollary:

\begin{corollarystar}
Let $K, K'$ be number fields. Suppose there exists a group homomorphism $\psi\colon \Hkab \to \Hkabx$ such that $L(\chi, s)$ divides $L(\psi(\chi), s)$ for every character $\chi \in \Hkab$. Then there exists a homomorphism $K \to K'$.
\end{corollarystar}

Moreover, we give a generalisation of \cite[Theorem 10.1]{CdSLMS}, proving that a single $L$-series can characterise the number field, even when division is allowed:

\begin{theoremA}\label{theoremA:existence}
Let $K$ be a number field. For any $l \geq 3$ there exists a character $\chi \in \Hkab[l]$ such that if for any number field $K'$ and any character $\chi' \in \Hkabx$ we have that $L(\chi',s)$ divides $L(\chi, s)$, then $K \simeq K'$.
\end{theoremA}

Note that as a corollary the $L$-series $L(\chi,s)$ and $L(\chi',s)$ are equal.

A natural follow-up question to this theorem that is not dealt with in this article is obtained by exchanging the role of $\chi$ and $\chi'$: does there exist a character $\chi \in \Hkab$ such that whenever there is a number field $K'$ and $\chi' \in \Hkabx$ such that $L(\chi, s)$ divides $L(\chi', s)$, the set $\Homfield$ is necessarily nonempty?

\section{Preliminaries}
Throughout the entire paper, we fix an algebraic closure $\overline{\Qz}$ of $\Qz$. By $K, K'$ and $N$ we denote number fields, where $N/\Qz$ is a Galois extension. We use $\mfp$ for the prime ideals of $K$ (henceforth we call prime ideals ``primes''), $\mfq$ for the primes of $K'$, and $\mfP$ for the primes of $N$. 

The set of primes of a number field $K$ is denoted $\cP_{K}$, and we denote the set of unramified primes of inertia degree one by $\pks$. We denote the inertia degree of $\mfp \in \cP_K$ by $f_{\mfp}$.

For a field $K$, let $G_K$ be its absolute Galois group, $K^\textrm{ab}$ be the composite of all its abelian extensions, and denote by $G_K^{\textrm{ab}}$ the Galois group of $K^\textrm{ab}/K$. The dual, denoted $\Hkab$, is the group of all Dirichlet characters (i.e.\ continuous homomorphisms) $G_K \to \Cz^\times$. For any prime number $l$ we denote by $\Hkab[l]$ the subgroup of $\Hkab$ generated by characters of order $l$.

\subsection*{Dirichlet characters} \leavevmode \\
For every Dirichlet character $\chi \in \Hkab$ there exists a unique finite cyclic extension $K_{\chi}/K$ of degree equal to the order of $\chi$ such that $\chi$ factors through $\Gal(K_{\chi}/K)$. Let $\mfp$ be a prime of $K$ unramified in $K_{\chi}/K$ and let $\Frob{\mfp}$ be the Frobenius element in $\Gal(K_{\chi}/K)$. We set $\chi(\mfp) := \chi(\Frob{\mfp})$. If $\chi$ has order $l$, then $\chi(\mfp)$ is an $l^\textup{th}$ root of unity. If $\mfp$ is a prime of $K$ that ramifies in $K_{\chi}/K$, we set $\chi(\mfp) = 0$.

\subsection*{The $L$-series of a Dirichlet character}
We give a brief definition of the $L$-series and formally state the definition of $L_p(\chi, T)$ and ``dividing'' $L$-series.
\begin{definition}
Associated to any Dirichlet character $\chi \in \Hkab$ is an $L$-series $L(\chi, s)$, defined by
\[
L(\chi, s) = \prod_{\mfp \in \cP_K} \big(1 - \chi(\mfp) p^{- f_{\mfp} s}\big)^{-1}.
\]
Here $s$ denotes a complex variable. $L(\chi, s)$ converges for any $s \in \mathbb{C}$ with $\textup{Re}(s) > 1$. 
\end{definition}

\begin{definition}\label{definition:local factor at p}
Denote $L_\mfp(\chi, T) := 1 - \chi(\mfp) T$. We define the \emph{local factor at $p$} as
\[
L_p(\chi, T) := \prod_{\substack{\mfp \in \cP_{K}\\ \mfp \mid p}} L_\mfp(\chi, T^{f_{\mfp}}).
\]
This is a polynomial in $T$ of degree at most $[K:\Qz]$. 
\end{definition}

\begin{definition}
Let $K, K'$ be number fields, and let $\chi \in \Hkab$, $\chi' \in \Hkabx$ be characters. We say that $L(\chi, s)$ \emph{divides} $L(\chi', s)$ if the polynomial $L_p(\chi, T)$ divides $L_p(\chi', T)$.
\end{definition}

\begin{remark}
There seems to be no clear analytic condition equivalent to the condition that $L(\chi,s)$ divides $L(\chi', s)$. For example, in Corollary~\ref{corollary:div_char} a character $\chi_N$ is constructed such that $\zeta_{\Qz}(s)$ divides $L(\chi_N, s)$. However, $\zeta_{\Qz}(s)$ has a pole at $s = 1$, while $L(\chi_N, s)$ does not.
\end{remark}

\section{Main theorem}
We create bijections between three sets of homomorphisms:
\begin{definition}\label{definition:hom_sets}
\mbox{ }
\begin{itemize}
\item $\HomLl$ is the set of homomorphisms $\psi\colon \Hkab[l] \to \Hkabx[l]$ for which $L_p(\chi, T)$ divides $L_p(\psi(\chi), T)$ for all rational primes $p$.
\item $\HomPld$ is the set of homomorphisms $\psi\colon \Hkab[l] \to \Hkabx[l]$ for which there exists a set $\cS \subseteq \pks$ of (Dirichlet) density one and a map $\phi\colon \pks \to \pkx$ such that $\mfp \cap \Zz = \phi(\mfp) \cap \Zz$ and
\[
\psi(\chi)(\phi(\mfp)) = \chi(\mfp)^{f_{\phi(\mfp)}}
\]
for any $\mfp \in \cS$ and all $\chi$ unramified at $\mfp$. We say that $\cS$ and $\phi$ are \emph{associated} to $\psi$. 
\item $\Homfield$ is the set of field homomorphisms $K \to K'$.
\end{itemize}
\end{definition}

\begin{theorem}\label{theorem:homs_in_bijection}
Let $K, K'$ be number fields, let $N$ be the Galois closure of the composite of $K$ and $K'$ and let $l > [N:\Qz]^2$ be a prime number. There exist injective maps
\begin{equation*}
\begin{tikzcd}
\HomLl \arrow{rr}{\Omega_2}&  & \arrow[dl, "\Omega_3"] \HomPld \\
 & \arrow{ul}{\Omega_1}\Homfield &  
\end{tikzcd}
\end{equation*}
such that $\Omega_3 \circ \Omega_2 \circ \Omega_1$ is the identity on $\Homfield$. As a result, the sets of homomorphisms are all in bijection.
\end{theorem}

\begin{remark}
Because the natural map between $\Homfield$ and $\HomPld$ is a bijection, for any $\psi \in \HomPld$ there is in fact a map $\phi\colon \pk \to \pkx$ (that is, defined on all primes of $K$) such that
\[
\psi(\chi)(\phi(\mfp)) = \chi(\mfp)^{f_{\phi(\mfp)}}
\]
for any $\mfp \in \pk$ and all $\chi$ unramified at $\mfp$.
\end{remark}

The remainder of the article consists of five sections: Section~\ref{section:comparison} briefly compares the recently introduced divisibility condition with other notions for the Dedekind zeta function that arise naturally from considering arithmetic equivalence. Section~\ref{section:iota} concerns some properties of the restriction map of characters induced from homomorphisms of fields. The results from this section are applied in Section~\ref{section:proofs}, where the main theorem is proven. In Section~\ref{section:counterexamples} we show that the main theorem does not hold if we remove the lower bound on $l$. Lastly, in Section~\ref{section:single} we prove that there exists $L$-series that cannot be properly divided by other $L$-series.

\section{Comparison of the divisibility condition}\label{section:comparison}
Let $K$ and $K'$ be number fields and let $N$ be a Galois extension containing both fields. Let $G = \Gal(N/\Qz)$, $H = \Gal(N/K)$ and $H' = \Gal(N/K')$, and denote $\rho = \textrm{Ind}^G_H(\mathbf{1}_H)$ and $\rho' = \textrm{Ind}^G_{H'}(\mathbf{1}_{H'})$.  The following properties are equivalent, see \cite[Section 1]{perlis}:
\begin{itemize}
\item $\zeta_K(s) = \zeta_{K'}(s)$;
\item $|c^G \cap H| = |c^G \cap H'|$ for every conjugacy class $c^G$;
\item $\text{charpoly}(\rho(c^G))= \text{charpoly}(\rho'(c^G))$ for every conjugacy class $c^G$;
\item $\rho \simeq \rho'$.
\end{itemize}

If $K'$ contains $K$ then $H'$ is a subgroup of $H$ and the following properties hold:
\begin{enumerate}
\item $\zeta_K(s)$ divides $\zeta_{K'}(s)$;
\item \label{enum:1} $|c^G \cap H| \geq |c^G \cap H'|$ for every conjugacy class $c^G$;
\item \label{enum:2} $\text{charpoly}(\rho(c^G)) \mid \text{charpoly}(\rho'(c^G))$ for every conjugacy class $c^G$;
\item \label{enum:3} $\rho$ is a subrepresentation of $\rho'$.
\end{enumerate}
Note that the first and third property are equivalent, as $\zeta_K(s) = L(1_H, s)$, and $L_p(1_H, T) = \text{charpoly}(\rho(c^G))$. 

Any of these properties might also occur when $H'$ is not a subgroup of $H$, and one might wonder whether any implications hold between these properties. Certainly property (\ref{enum:3}) implies property (\ref{enum:2}), as the characteristic polynomial of a representation is the product of the characteristic polynomials of its decomposition. However, none of the other implications hold, as the following examples show:
\begin{itemize}
\item (\ref{enum:2}) $\not\Rightarrow$ (\ref{enum:3}). Let $G = S_4$, $H = A_4$, and $H' = \langle (12), (123) \rangle \simeq S_3$. The representation $\rho$ decomposes as the sum of the trivial and the sign representation, whilst $\rho'$ is the direct sum of the trivial and the standard representation, which we can check using Frobenius reciprocity:
\[
\langle \chi_{\rho'}, \chi_\text{standard} \rangle_G = \langle 1_{H'}, \text{Res}(\chi_\text{standard}) \rangle_{H'} = \frac12 (1 \cdot 3 + 1 \cdot -1) = 1.
\]
We can also write down all characteristic polynomials:
\begin{center}
\begin{tabular}{ l |r| r }
  $c^G$ & $\text{charpoly}(\rho(c^G))$ & $\text{charpoly}(\rho'(c^G))$ \\ \hline
  $()$ & $(1-T)^4$ & $(1-T)^2$ \\
  $(..)$ & $(1-T)^2(1-T^2)$ & $(1-T^2)$\\
  $(...)$ & $(1-T^2)^2$ & $(1-T)^2$ \\
  $(..)(..)$ & $(1-T)(1-T^3)$ & $(1-T)^2$\\
  $(....)$ & $(1-T^4)$ & $(1-T^2)$
\end{tabular}
\end{center}
Note that $\text{charpoly}(\rho'(c^G))$ divides $\text{charpoly}(\rho(c^G))$ for all conjugacy classes $c^G$.

\item (\ref{enum:3}) $\not\Rightarrow$ (\ref{enum:1}). Let $G = S_4$, $H = \langle (12), (123) \rangle \simeq S_3$, and $H' = \langle (12)(34) \rangle \simeq C_2$. This triple does not have property (\ref{enum:1}) because $H$ does not contain any element of the conjugacy class $(..)(..)$. The induced representation $\rho'$ decomposes, as before, into the direct sum of the trivial representation and the standard representation of $S_4$. However, $\rho$ also contains the standard representation: 
\[
\langle \chi_\rho, \chi_\text{standard} \rangle_G = \langle 1_H, \text{Res}(\chi_\text{standard}) \rangle_H = \frac16 (1 \cdot 3 + 3 \cdot 1 + 2 \cdot 0) = 1.
\]
Hence property (\ref{enum:3}) holds for this choice of $G$, $H$, and $H'$.
\item (\ref{enum:1}) $\not\Rightarrow$ (\ref{enum:2}). The following counterexample was found using GAP. Let $G$ be the first nonsplit extension of $C_4^2$ by $D_4$ (this group has groupID $128, 134$), i.e.
\begin{align*}
G = \langle a,b,c,d | \,&a^4=b^4=c^4=1,d^2=a^{-1},ab=ba,cac^{-1}=a^{-1}b,ad=da, \\ & cbc^{-1}=a^2b, dbd^{-1}= a^2b^{-1},dcd^{-1}=a^{-1}c^{-1}\rangle.
\end{align*}
Let $H = \langle b^{-2}, ac^2, ad^{-1}c^{-1} \rangle \simeq D_4$ and $H' = \langle ac^2, a^{-1} d c^{-1} a \rangle \simeq C_2 \times C_2$. Both are contained in the union of $e^G, (ac^2)^G, (b^2)^G, (b^3c^3d)^G$, and $(a^2b^3c^3d)^G$, and we have

\begin{center}
\begin{tabular}{ l | c c r r }
  $c^G$ & $|c^G \cap H|$ & $|c^G \cap H'|$ & $\text{charpoly}(\rho(c^G))$ & $\text{charpoly}(\rho'(c^G))$ 
  \\ \hline
  $e^G$ 			& 1 & 1 & $(1-T)^{16}$ 					& $(1-T)^{32}$ \\
  $(ac^2)^G$ 		& 2 & 1 & $(1-T)^4(1-T^2)^6$ 			& $(1-T)^4(1-T^2)^{14}$ \\
  $(b^2)^G$ 		& 1 & 0 & $(1-T)^{16}$ 					& $(1-T^2)^{16}$ \\
  $(b^3c^3d)^G$ 	& 2 & 2 & $(1-T)^4(1-T^2)^6$ 			& $(1-T)^8(1-T^2)^{12}$ \\
  $(a^2b^3c^3d)^G$ 	& 2 & 0 & $(1-T)^4(1-T^2)^6$ 			& $(1-T^4)^8$.
\end{tabular}
\end{center}
Thus property (\ref{enum:1}), but not property (\ref{enum:2}), holds for this choice of $G$, $H$, and $H'$.
\end{itemize}

\section{Properties of the restriction map}\label{section:iota}
An important tool in the proof of the main theorem is the restriction map $\iota$, induced by an inclusion of fields:

\begin{definition}
Let $K, K'$ be number fields for which $K \subseteq K'$. Let $\iota_{K, K'}\colon \Hkab \to \Hkabx$ denote the restriction map, i.e. $\iota_{K, K'}(\chi) = \restr{\chi}{G_{K'}}$.
\end{definition}

If a character $\chi \in \Hkab$ is associated to the extension $K_\chi/K$, then $\iota_{K, K'}(\chi)$ is associated to the extension $K'K_\chi/K'$. Moreover, if $\chi$ is unramified at a prime $\mfp$ of $K$, then $\iota_{K, K'}$ is unramified at any prime $\mfP$ of $K'$ lying over $\mfp$, and 
\[
\chi(\mfp)^{f_{\mfP}/f_{\mfp}} = \iota_{K, K'}(\chi)(\mfP).
\]
As a result, $L_{\mfp}(\chi, T^{f_{\mfp}})$ divides $L_{\mfP}(\iota_{K, K'}(\chi), T^{f_{\mfP}})$. If $\chi$ is ramified at $\mfp$, then $L_{\mfp}(\chi, T^{f_{\mfp}}) = 1$, hence $L_{\mfp}(\chi, T^{f_{\mfp}})$ divides $L_{\mfP}(\iota_{K, K'}(\chi), T^{f_{\mfP}})$ as well. Therefore, for any prime $\mfp$ of $K$ and any prime $\mfP$ of $K'$ lying over $\mfp$, $L_{\mfp}(\chi, T^{f_{\mfp}})$ divides $L_{\mfP}(\iota_{K, K'}(\chi), T^{f_{\mfP}})$. This implies the following:

\begin{lemma}\label{lemma:iota_in_HomL}
Let $K, K'$ be number fields such that $K$ is contained in $K'$. Then $\iota_{K, K'} \in \HomL$. \qed
\end{lemma}

We end this section with two lemmas that will be useful in the next section.

\begin{lemma}\label{lemma:iota_injective}
For any prime $l > [K':K]$ the map $\iota_{K, K'}\colon \Hkab[l] \to \Hkabx[l]$ is injective.
\end{lemma}

\begin{proof}
The character $\iota_{K, K'}(\chi)$ is trivial precisely when $K_{\chi}$ is contained in $K'$. However, if $\chi$ is not the trivial character, then $[K_{\chi}:K] = l > [K':K]$,  hence $\iota_{K, K'}(\chi)$ is not trivial. Thus $\iota_{K, K'}\colon \Hkab[l] \to \Hkabx[l]$ is injective.
\end{proof}

\begin{lemma}\label{lemma:KinK'}
Let $K, K'$ and $N$ be number fields such that $K, K' \subseteq N$ and $N/\Qz$ is Galois. Let $l > [N: K']$ be a prime number. Assume there exists a map $\alpha\colon \Hkab[l] \to \Hkabx[l]$ such that
\[
\iota_{K',N} \circ \alpha = \iota_{K, N}.
\]
Then $K \subseteq K'$ and $\alpha = \iota_{K, K'}$. 
\end{lemma}

\begin{proof}
Let $p$ be a prime that is totally split in $N/\Qz$, and let $\mfq$ be a prime of $K'$ lying over $p$. We show that all primes of $N$ that lie over $\mfq$ lie over the same prime of $K$.

Denote by $\mfP_1, \dots, \mfP_k$ the primes of $N$ lying over $\mfq$. Let $\mfp_i := \mfP_i \cap \cO_{K}$. We argue by contradiction that all $\mfp_i$ are the same prime. Suppose for some $i,j$ we have $\mfp_i \neq \mfp_j$. By the Grunwald-Wang theorem \cite[Ch.\ X, Thm.\ 5]{artintate} there exists a character $\chi \in \Hkab[l]$ for which $\chi(\mfp_i) \neq \chi(\mfp_j)$. 

Any character in the image of $\iota_{K',N} \circ \alpha$ has the same value at all $\mfP_1, \dots, \mfP_k$: we have 
\begin{align*}
(\iota_{K',N} \circ \alpha)(\chi)(\mfP_i) &= \alpha(\chi)(\mfP_i \cap \cO_{K'})\\
										  &= \alpha(\chi)(\mfq)\\
										  &= \alpha(\chi)(\mfP_j \cap \cO_{K'})\\
										  &= (\iota_{K',N} \circ \alpha)(\chi)(\mfP_j),
\end{align*}
whilst 
\[
\iota_{K, N}(\chi)(\mfP_i) = \chi(\mfp_i) \neq \chi(\mfp_j) = \iota_{K, N}(\chi)(\mfP_j),
\]
which contradicts $\iota_{K',N} \circ \alpha = \iota_{K, N}$. Hence every $\mfP_i$ lies over the same prime of $K$.

Let $p$ a prime that is totally split in $N/\Qz$, $\mfp$ a prime of $K$ lying over $p$ and $\mfP$ a prime of $N$ lying over $\mfp$. Denote $\mfq := \mfP \cap \cO_{K'}$, and let $\sigma \in \Gal(N/K')$. The prime $\sigma(\mfP)$ lies over $\mfq$ as $\sigma$ fixes $K'$. By the previous argument, $\sigma(\mfP)$ also lies over $\mfp$. As this holds for any $\mfP$ lying over $\mfp$, $\sigma$ permutes the primes lying over $\mfp$. Using \cite[Lemma 6.6]{GDTHJS}, this implies that $\sigma$ restricts to an automorphism of $K$. As $\restr{\sigma}{K}(\mfp) = \mfp$, from \cite[Lemma 6.2]{GDTHJS} we find that $\sigma$ is the identity on $K$, hence $\sigma \in \Gal(N/K)$. Now $K \subseteq K'$ follows from $\Gal(N/K') \subseteq \Gal(N/K)$. 

Because $K \subseteq K'$, we have $\iota_{K, N} = \iota_{K', N} \circ \iota_{K, K'}$. As $l > [N:K']$, Lemma~\ref{lemma:iota_injective} states that the map $\iota_{K', N}$ is injective, hence we find from $\iota_{K',N} \circ \alpha = \iota_{K, N}$ that $\alpha = \iota_{K, K'}$.
\end{proof}

\section{Proofs}\label{section:proofs}
\subsection{The map $\Omega_1$}
Any $\tau \in \Homfield$ produces a field $K_\tau = \im(\tau) \subseteq K'$ and an isomorphism $\tau\colon K \DistTo K_\tau$. As a result (see also \cite[Sec.\ 6]{GDTHJS}) we have an induced bijection of primes $\tau\colon \cP_{K} \to \cP_{K_\tau}$ and a group isomorphism $\psi_\tau\colon \Hkab \DistTo \Hktauab$ for which 
\[
\psi_\tau(\chi)(\tau(\mfp)) = \chi(\mfp)
\]
for every prime $\mfp$ of $K$ and every $\chi \in \Hkab$. It follows that $L(\chi, s) = L(\psi_\tau(\chi), s)$ for every $\chi \in \Hkab$.

\begin{lemma}\label{lemma:Omega1_injective}
The map $\Omega_1\colon \Homfield \to \HomL$ given by $\tau \mapsto \iota_{K_\tau, K'} \circ \psi_\tau$ is well-defined and injective. 
\end{lemma}

\begin{proof}
As $\tau\colon K \DistTo K_\tau$ is an isomorphism, we have $L(\chi, s) = L(\psi_\tau(\chi), s)$. By Lemma~\ref{lemma:iota_in_HomL}, $\iota_{K_\tau, K'} \in \HomLtau$, hence $\iota_{K_\tau, K'} \circ \psi_\tau$ is an element of $\HomL$. To prove injectivity, suppose that for some $\sigma, \tau \in \Homfield$ we have
\[
\iota_{K_\sigma, K'} \circ \psi_\sigma = \iota_{K_\tau, K'} \circ \psi_\tau.
\]
Then $\iota_{K_\sigma, K'} \circ (\psi_\sigma \circ \psi_\tau^{-1}) = \iota_{K_\tau, K'}$, hence by Lemma~\ref{lemma:KinK'} we have $K_\tau \subseteq K_\sigma$ and $\psi_\sigma\circ \psi_\tau^{-1} = \iota_{K_\tau, K_\sigma}$. As $K_\tau$ and $K_\sigma$ are both conjugates of $K$ (over $\Qz$) we conclude that $K_\tau = K_\sigma$ and that $\psi_\sigma\circ \psi_\tau^{-1} = \iota_{K_\tau, K_\sigma} = \restr{\textup{id}}{K_\tau}$, hence $\psi_\sigma = \psi_\tau$. Using \cite[Lemma 6.4]{GDTHJS} we conclude that $\sigma = \tau$ as maps $K \DistTo K_\tau$, hence also as elements of $\Homfield$. 
\end{proof}

\subsection{The map $\Omega_2$}
The aim of this section is to prove the following theorem:

\begin{theorem}\label{theorem:HomL_to_HomP}
Let $l> [K':\Qz]$ be a prime number, and suppose that $\psi$ is an element of $\HomLl$. Then $\psi \in \HomPld$.
\end{theorem}

\begin{corollary}
The map $\Omega_2\colon \HomLl \to \HomPld$ given by $\psi \mapsto \psi$ is well-defined and injective.
\end{corollary}

We prove a triplet of lemmas from which the theorem readily follows. Fix a prime $l > [K':\Qz]$ and a homomorphism $\psi \in \HomLl$.

Let $\pow(\pkx)$ denote the power set of $\pkx$ and define the map $\varphi\colon \Hkab[l] \times \pks \to \pow(\pkx)$ as follows: $\varphi(\chi, \mfp)$ is the set of primes $\mfq $ of $K'$ such that $\mfq \cap \Zz = \mfp \cap \Zz$ and 
\[
\psi(\chi)(\mfq) = \chi(\mfp)^{f_{\mfq}}.
\]

\begin{lemma}\label{lemma:multiplicative_properties_of_phi}
Fix a prime $\mfp \in \pks$. The following hold:
\begin{enumerate}
\item If $\chi \in \Hkab[l]$ is unramified at $\mfp$, then $\varphi(\chi, \mfp) = \varphi(\chi^j, \mfp)$ for any $j$ with $(j, l) = 1$.
\item Suppose $\mfq \in \varphi(\chi_1, \mfp)$. Then, provided that $\chi_1$ and $\chi_2$ are unramified at $\mfp$,
\[
\mfq \in \varphi(\chi_2, \mfp) \Longleftrightarrow \mfq \in \varphi(\chi_1 \chi_2, \mfp).
\]
\end{enumerate}
\end{lemma}

\begin{proof}
All characters concerned are unramified at $\mfp$, hence the value at $\mfp$ of the product of the characters equals the product of the values at $\mfp$ of the individual characters.
\begin{enumerate}
\item As $\psi$ is a group homomorphism, $\psi(\chi^j) = \psi(\chi)^j$. Note that as $\psi(\chi)$ is of order $l$ and $(j, l) = 1$, we have that $\psi(\chi)$ is ramified at a prime precisely when $\psi(\chi)^j$ is. The result follows from
\[
\psi(\chi)(\mfq) = \chi(\mfp)^{f_{\mfq}} \Longleftrightarrow \psi(\chi)(\mfq)^j = \chi(\mfp)^{j f_{\mfq}}.
\]
\item We have $\psi(\chi_1 \chi_2) = \psi(\chi_1) \psi(\chi_2)$. Moreover, as $\chi_1$ is unramified at $\mfp$, and $\mfq \in \varphi(\chi_1, \mfp)$, it follows that $\psi(\chi_1)$ is unramified at $\mfq$. Hence we have
\[
\psi(\chi_1\chi_2)(\mfq) = \psi(\chi_1)(\mfq) \psi(\chi_2)(\mfq) = \chi_1(\mfp)^{f_{\mfq}} \psi(\chi_2)(\mfq),
\]
hence $\psi(\chi_1\chi_2)(\mfq) = \chi_1\chi_2(\mfp)^{f_{\mfq}}$ if and only if $\psi(\chi_2)(\mfq) = \chi_2(\mfp)^{f_{\mfq}}$.\qedhere
\end{enumerate}
\end{proof}

\begin{lemma}\label{lemma:phi_nonempty}
For any $\chi \in \Hkab$ and $\mfp \in \pks$ such that $\chi$ is unramified at $\mfp$, the set $\varphi(\chi, \mfp)$ is non-empty.
\end{lemma}

\begin{proof}
Denote $p = \mfp \cap \Zz$. As $L_p(\chi, T) \mid L_p(\psi(\chi), T)$ and $L_{\mfp}(\chi, T) = 1 - \chi(\mfp)T$, the polynomial $L_p(\psi(\chi), T)$ has a zero at $T = \chi(\mfp)^{-1}$.

Hence there is a prime $\mfq$ of $K'$ lying over $p$ such that $L_{\mfq}(\psi(\chi), T^{f_{\mfq}})$ has a zero at $T = \chi(\mfp)^{-1}$. As 
\[
L_{\mfq}(\psi(\chi), T^{f_{\mfq}}) = 1 - \psi(\chi)(\mfq) T^{f_{\mfq}},
\]
we have 
\[
\psi(\chi)(\mfq) = \chi(\mfp)^{f_{\mfq}},
\]
and thus $\mfq \in \varphi(\chi, \mfp)$.
\end{proof}

\begin{lemma}\label{lemma:phi_intersection_nonempty}
Let $l > [K':\Qz]$ be a prime number. For a fixed $\mfp \in \pks$ the set
\[
\bigcap_{\chi(\mfp) \neq 0} \varphi(\chi, \mfp)
\]
is non-empty, where the intersection is taken over all characters $\chi \in \Hkab[l]$ unramified at $\mfp$.
\end{lemma}

\begin{proof}
As all $\varphi(\chi, \mfp)$ contain only primes lying over $\mfp \cap \Zz$ and are therefore finite, it suffices to show that any finite intersection is non-empty for any choice of characters unramified at $\mfp$. We proceed by induction.

\noindent\textbf{Induction hypothesis.} Suppose that for a certain $n \geq 1$ and any choice of characters $\chi_1, \dots, \chi_n \in \Hkab[l]$ unramified at $\mfp$ we have that
\[
\varphi(\chi_1, \mfp) \cap \varphi(\chi_2, \mfp) \cap \dots \cap \varphi(\chi_n, \mfp) \neq \emptyset.
\]
The case $n = 1$ holds by Lemma~\ref{lemma:phi_nonempty}.

\noindent\textbf{Induction step.}  Let $\chi_1, \chi_2, \dots, \chi_{n+1} \in \Hkab[l]$ be characters unramified at $\mfp$. For any integer $j$, define
\begin{align*}
S_j := \varphi(\chi_1^j\chi_2, \mfp) \cap \varphi(\chi_3, \mfp) \cap \dots \cap \varphi(\chi_{n+1}, \mfp)
\end{align*}
By the induction hypothesis, $S_j$ is non-empty for any $j \in \Nz_0$. As every $S_j$ consists only of primes of $K'$ lying over $\mfp \cap \Zz$ of which there are at most $[K':\Qz]$, there exist integers $0 \leq j < k \leq [K':\Qz]$ for which the intersection $S_j \cap S_k$ is nonempty. Let $\mfq \in S_j \cap S_k$, then $\mfq \in \varphi(\chi_1^j\chi_2, \mfp)$ and $\mfq \in \varphi(\chi_1^k\chi_2, \mfp)$. We prove that $\mfq \in \varphi(\chi_1, \mfp) \cap \varphi(\chi_2, \mfp)$ by repeated application of Lemma~\ref{lemma:multiplicative_properties_of_phi}.

As $\mfq \in \varphi(\chi^j \chi_2, \mfp)$ we have $\mfq \in \varphi(\chi_1^{-j}\chi_2^{-1}, \mfp)$. We combine this with $\mfq \in \varphi(\chi_1^k\chi_2, \mfp)$ to obtain $\mfq \in \varphi(\chi_1^k\chi_2 \chi_1^{-j}\chi_2^{-1}, \mfp)$, i.e.\ $\mfq \in \varphi(\chi_1^{k-j}, \mfp)$. By assumption, $l > [K':\Qz] > k-j > 0$, thus $\mfq \in \varphi(\chi_1, \mfp)$. Furthermore, $\mfq \in \varphi(\chi_1^{-k}, \mfp)$, hence it follows from $\mfq \in \varphi(\chi_1^k\chi_2, \mfp)$ that $\mfq \in \varphi(\chi_1^k\chi_2 \chi_1^{-k}, \mfp) = \varphi(\chi_2, \mfp)$.
\end{proof}

\begin{proof}[Proof of Theorem~\ref{theorem:HomL_to_HomP}]
Lemma~\ref{lemma:phi_intersection_nonempty} shows that for any $\mfp \in \pks$ there is a prime $\mfq$ of $K'$ lying over $\mfp \cap \Zz$ such that 
\[
\psi(\chi)(\mfq)^{f_{\mfq}} = \chi(\mfp)
\]
for any $\chi$ unramified at $\mfp$. Setting $\varphi(\mfp) := \mfq$ we find $\psi \in \HomPld$.
\end{proof}

\subsection{The map $\Omega_3$}
Recall that any $\tau \in \Homfield$ induces an isomorphism $\psi_{\tau}\colon \Hkab[l] \DistTo \Hktauab[l]$. We prove the following theorem:

\begin{theorem}\label{theorem:HomP_to_Homfield}
Let $K, K'$ be number fields and let $N$ be the Galois closure of $KK'$. Let $l > [N:\Qz]^2$ be a prime number and let $\psi \in \HomPld$. Then there exists a $\tau \in \Homfield$ such that $K_\tau \subseteq K'$, and $\psi$ is of the form $\iota_{K_\tau, K'} \circ \psi_\tau$.
\end{theorem}

\begin{remark}
By injectivity of $\Omega_1$(Lemma~\ref{lemma:Omega1_injective}) there is a unique $\tau$ such that $\psi = \iota_{K_\tau, K'} \circ \psi_\tau$. 
\end{remark}

This theorem provides the desired map $\Omega_3$:

\begin{corollary}
Let $\tau$ be the unique homomorphism such that $\psi = \iota_{K_\tau, K'} \circ \psi_\tau$. The map $\Omega_3\colon \HomPld \to \Homfield$ given by $\psi \mapsto \tau$, is well-defined and injective.
\end{corollary}

\begin{proof}
Well-definedness follows from Theorem~\ref{theorem:HomP_to_Homfield}. Injectivity is almost immediate: suppose $\psi_1, \psi_2 \in \HomPld$ have the same image $\tau$. Then
\[
\psi_1 = \iota_{K_\tau, K'} \circ \psi_\tau = \psi_2.\qedhere 
\]
\end{proof}

We first prove Theorem~\ref{theorem:HomP_to_Homfield} in the case where $K \subseteq K'$ and $K'/\Qz$ is Galois, and then deduce the general case.

\begin{lemma}\label{lemma:hom_version_galois}
Let $K, K'$ be number fields such that $K \subseteq K'$ and $K'/\Qz$ is Galois. Let $l > [K':\Qz]^2$ be a prime number, let $\psi \in \HomPld$ and let $\cS \subseteq \pks$ be a set of primes of $K$ of density one. Let $\varphi\colon \cS \to \pkx$ be a map of prime ideals associated to $\psi$ (as in Definition~\ref{definition:hom_sets}). Then $\psi$ is of the form $\iota_{K_\tau, K'} \circ \psi_\tau$, where $\tau \in \Homfield$. 
\end{lemma}

\begin{proof}
Let $\cS'$ denote the set of primes of $K'$ that lie over a prime in $\cS$, and let $\cS'_p$ denote the primes in $\cS'$ that lie over the rational prime $p$. As for any $\mfp \in \cS$ we have $\mfp \cap \Zz = \varphi(\mfp) \cap \Zz$, it follows that if $\cS'_p$ is nonempty, then the intersection of $\im(\varphi)$ and $\cS'_p$ is nonempty.

The set $\pks - \cS$ has density zero, hence $\pksx - \cS'$ has density zero as well, i.e.\ $\cS'$ has density one. Because any $\cS'_p$ has size at most $[K':\Qz]$ and all primes in $\cS'_p$ have the same norm $p$, the image of $\varphi$ has density at least $1/[K':\Qz]$. For any $\sigma \in \Gal(K'/\Qz)$, let 
\[
\im(\varphi)_\sigma = \{\varphi(\mfp): \mfp \in \cS, \sigma^{-1}(\varphi(\mfp)) \textup{ lies over } \mfp \}.
\]
The group $\Gal(K'/\Qz)$ acts transitively on the primes of $K'$ lying over $\Qz$, thus for any prime $\mfp$ of $K$ there is a $\sigma \in \Gal(K'/\Qz)$ such that $\sigma^{-1}(\varphi(\mfp))$ lies over $\mfp$, hence
\[
\bigcup_{\sigma \in \Gal(K'/\Qz)} \im(\varphi)_\sigma = \im(\varphi).
\]
The size of $\Gal(K'/\Qz)$ equals $[K':\Qz]$, hence there is a $\sigma \in \Gal(K'/\Qz)$ for which $\im(\varphi)_\sigma$ is not contained in any set of density more than $1/[K':\Qz]$ times the density of $\im(\varphi)$, hence a density more than $1/[K':\Qz]^2$ in the primes of $K'$. Fix a $\sigma \in \Gal(K'/\Qz)$ for which this holds, and let $\chi \in \Hkab[l]$. For any $\varphi(\mfp) \in \im(\varphi)_\sigma$ such that $\chi$ is unramified at $\mfp$, we have that $\sigma^{-1}(\varphi(\mfp))$ lies over $\mfp$, hence
\[
\psi(\chi)(\varphi(\mfp)) = \chi(\mfp)^{f_{\varphi(\mfp)}} = \iota_{K, K'}(\chi)(\sigma^{-1}(\varphi(\mfp))) = (\psi_\sigma \circ \iota_{K, K'})(\varphi(\mfp)).
\]
As a result, the set of primes of $K'$ on which the character $(\psi_\sigma \circ \iota_{K, K'}) \psi^{-1}(\chi)$ has value $1$ is not contained in a set of density more than $1/[K':\Qz]^2$. From the Chebotarev density theorem \cite[Ch.\ VII, \S 13, p.\ 545]{neukirch2013algebraic} it follows that $(\psi_\sigma \circ \iota_{K, K'})(\psi^{-1}(\chi))$ cannot have order $l > [K':\Qz]^2$, hence it is the trivial character. We conclude that $\psi = \psi_\sigma \circ \iota_{K, K'}$.

Let $\tau = \restr{\sigma}{K}$. We finish the proof by showing that $\psi_\sigma \circ \iota_{K, K'} = \iota_{K_\tau, K'} \circ \psi_\tau$. Let $\chi \in \Hkab[l]$, and let $\mfq$ be any prime of $K'$ with $f_{\mfq} = 1$ such that $\chi$ is unramified at all primes of $K$ lying over $\mfq \cap \Zz$. Note that 
\begin{align*}
(\psi_\sigma \circ \iota_{K, K'})(\chi)(\mfq) &= \iota_{K, K'}(\chi)(\sigma^{-1}(\mfq)) \\
										 &= \chi(\sigma^{-1}(\mfq) \cap \cO_K) 
\end{align*}
and
\begin{align*}
(\iota_{K_\tau, K'} \circ \psi_\tau)(\chi)(\mfq) &= \psi_\tau(\chi)(\mfq \cap \cO_{K_\tau}) \\
										 &= \chi(\tau^{-1}(\mfq \cap \cO_{K_\tau}))\\
										 &= \chi(\sigma^{-1}(\mfq) \cap \cO_{K}).			
\end{align*}
As the primes of inertia degree one form a set of primes of density one and $\chi$ is ramified at only finitely many primes, it follows by the Chebotarev density theorem that $(\psi_\sigma \circ \iota_{K, K'})(\chi) = (\iota_{K_\tau, K'} \circ \psi_\tau)(\chi)$. As $\chi$ was chosen arbitrarily, we conclude that $\psi = \psi_\sigma \circ \iota_{K, K'} = \iota_{K^\sigma, K'} \circ \psi_\tau$.
\end{proof}

\begin{proof}[Proof of Theorem~\ref{theorem:HomP_to_Homfield}]
Let $N$ be the Galois closure of $KK'$ over $\Qz$, and let $\phi$ and $\cS$ be associated to $\psi$. We claim that $\iota_{K', N} \circ \psi\colon \Hkab \to \Hnab$ is an element of $\HomPNl$. Indeed, for any $\mfp \in \cS$, any $\chi$ unramified at $\mfp$, and any $\mfP$ lying over $\varphi(\mfp)$ we have
\[
(\iota_{K', N} \circ \psi)(\chi)(\mfP) = \psi(\chi)(\varphi(\mfp))^{f_{\mfP}/f_{\varphi(\mfp)}} = \chi(\mfp)^{f_{\mfP}}.
\]
As $l > [N:\Qz]^2$, we know by Lemma~\ref{lemma:hom_version_galois} that there exists a $\tau \in  \HomfieldN$ such that
\[
\iota_{K', N} \circ \psi = \iota_{K_\tau, N} \circ \psi_\tau.
\] 

By Lemma~\ref{lemma:KinK'} (applied to $K_\tau, K$, $N$, and $\psi \circ \psi_\tau^{-1}$) we find $K_\tau \subseteq K'$ and $\psi \circ \psi_\tau^{-1} = \iota_{K_\tau, K'}$, i.e.\ $\psi = \iota_{K_\tau, K'} \circ \psi_\tau$.
\end{proof}

\section{Counterexamples for small $l$}\label{section:counterexamples}
The goal of this section is to show that Theorem~\ref{theoremA:bijection_of_homs} does not hold without any restrictions on $l$ in terms of the degrees of the fields; in particular, we construct an explicit counterexample for any prime number $l \geq 5$. Let $N/\Qz$ be a Galois extension with Galois group $S_l$. Take any cyclic subgroup $C_l$ of order $l$ of $S_l$ and let $K= N^{C_l}$.

\begin{lemma}\label{lemma:prime_above_totally_split} 
For any rational prime $p$ in $N/\Qz$ there is a prime $\mfp$ lying over $p$ of $K$ that is totally split in $N/K$. 
\end{lemma}

\begin{proof}
(This proof is due to Hendrik Lenstra and Peter Stevenhagen.) Denote the Galois conjugates of $K$ by $K = K_1, \dots, K_m$. Let $p$ be any rational prime. For any prime $\mfq$ in $N$ lying over $p$, consider the decomposition groups $D_{\mfq, K_i} \subseteq \Gal(N/K_i)$ for $i = 1, \dots, m$. If any of these decomposition groups is trivial, say $D_{\mfq, K_i}$, then $\mfq \cap K_i$ splits completely in $N/K_i$. As $K$ and $K_i$ are isomorphic, there is also a prime $\mfp$ in $K$ that splits completely in $N/K$.

If none of the decomposition groups are trivial, then $D_{\mfq, K_i} = \text{Gal}(N/K_i)$ as $[N:K_i]$ is prime. The decomposition group $D_{\mfq, \mathbb{Q}}$ contains every $D_{\mfq, K_i} = \text{Gal}(N/K_i)$. Any $l$-cycle in $S_l$ is contained in one of the $\Gal(N/K_i)$, hence $D_{\mfq, \mathbb{Q}}$ contains a subgroup isomorphic to $A_l$. This implies that $D_{\mfq, \mathbb{Q}}$ is isomorphic to either $A_l$ or $S_l$. However, both $A_l$ and $S_l$ are not solvable (and decomposition groups should be), hence this is a contradiction. 
\end{proof}

Let $\chi_N \in \Hkab[l]$ be the character associated to the cyclic degree $l$ extension $N/K$. Any prime $\mfp$ of $K$ splits completely in $N/K$ precisely when $\chi_N(\mfp) = 1$, hence the following holds:

\begin{corollary}\label{corollary:div_char}
The character $\chi_N$ has the property that for any rational prime $p$ there exists a prime $\mfp \mid p$ of $K$ such that $\chi_N(\mfp) = 1$. \qed
\end{corollary}  

For any rational prime $p$, denote by $\varphi(p)$ the prime $\mfp$ lying over $p$ obtained from this corollary. We show that $\HomfieldQ$ and $\HomQL$ are not in bijection.  Choose a set $\{\chi_1, \chi_2, \dots\} \subseteq \Hqab[l]$ of independent generators of $\Hqab[l]$ (chosen, for example, to correspond to pairwise disjoint cyclic degree $l$ extensions of $\Qz$ whose composite is the composite of all cyclic degree $l$ extensions of $\Qz$). Define the map $\psi\colon \Hqab[l] \to \Hkab[l]$ by setting 
\[
\psi(\chi_i) = \iota_{\Qz, K}(\chi_i) \cdot \chi_{N}
\]
and extending this multiplicatively: every $\chi \in \Hqab[l]$ can be written as a finite product $\prod\limits_{i=1}^k \chi_{a_i}$, with $a_1, \dots, a_k$ positive integers, and so $\psi(\chi) = \iota_{\Qz, K}(\chi) \cdot \chi^k_N$. We proceed by showing that $\psi \in \HomQL$. Note that $L_p(\chi, T) = 1- \chi(p) T$.

If $p$ is a rational prime at which $\chi$ is unramified, we have
\begin{align*}
\psi(\chi)(\varphi(p)) &= (\iota_{\Qz, K}(\chi) \cdot \chi_{N})(\varphi(p)) 
                    \\&= \iota_{\Qz, K}(\chi)(\varphi(p)) \cdot \chi_{N}(\varphi(p)) 
                    \\&= \iota_{\Qz, K}(\chi)(\varphi(p)) 
                    \\&= \chi(p)^{f_{\varphi(p)}}.
\end{align*}
It follows that $\psi \in \HomQPld$. Moreover the polynomial $L_p(\psi(\chi), T)$ has $1 - \chi(p)^{f_{\varphi(p)}} T^{f_{\varphi(p)}}$ as a factor, which is divisible by $L_p(\chi, T) = 1- \chi(p) T$. If $\chi$ is ramified at $p$ then $L_p(\chi, T) = 1$ and hence certainly divides $L_p(\psi(\chi), T)$. Therefore $L(\chi, s)$ divides $L(\psi(\chi), s)$ and therefore $\psi \in \HomQL$, but $\psi$ is not the identity map. We conclude that $\HomfieldQ$ and $\HomQL$ are not in bijection.

\section{Characterizing a number field with a single $L$-series}\label{section:single}
We use the following setup and notation: let $l \geq 3$ be a prime number, $K$ a number field of degree $n$, $N$ a finite extension of $K$ that is Galois over $\Qz$, let $\zeta = \exp(2 \pi i / l)$ and let $\widetilde{\chi} \in \Hkab[l]$ be the character specified in \cite[Theorem 10.1]{CdSLMS}, i.e.\ there is a rational prime $p$ that splits completely into $\mfp_1, \dots, \mfp_n$ in $K/\Qz$ such that $\widetilde{\chi}(\mfp_1) = \zeta$ and $\widetilde{\chi}(\mfp_i) = 1$ for all $2 \leq i \leq n$. 

Write $G := \Gal(N/\Qz)$, $H := \Gal(N/K)$, and $C = \langle \zeta \rangle = \Gal(K_{\widetilde{\chi}}/K)$. Let $g_1, g_2, \dots, g_n \in G$ be representatives of the cosets in $G/H$, where $g_1 = e$. Let $M$ denote the Galois closure of $K_{\widetilde{\chi}} / \Qz$. Now \cite[Theorem 9.1]{CdSLMS} shows that $\Gal(M/\Qz) \simeq C^n \rtimes G$, where $G$ acts on $C^n$ by permuting the copies of $C$ in the same way $G$ permutes the cosets $G/H$. Moreover, we have $\Gal(M/K) \simeq C^n \rtimes H$. The character $\widetilde{\chi}$ factors through $\Gal(M/K)$ and maps an element $(a_1, \dots, a_n, h) \in C^n \rtimes H$ to $a_1$. 

In view of \cite[Theorem 10.1]{CdSLMS} it is sufficient to prove the following theorem:

\begin{theorem}\label{theorem:dividing_L_series_implies_equal_degree}
Let $K'$ be a number field of degree $n'$ and $\chi' \in \Hkabx$ a character such that $L(\chi', s)$ divides $L(\widetilde{\chi},s)$. Then $n = n'$. 
\end{theorem}

\begin{proof}[Proof of Theorem~\ref{theoremA:existence}] If $n = n'$, then for all but finitely many $p$ (namely those unramified in $KK'$), the local factors $L_p(\widetilde{\chi}, T)$ and $L_p(\chi',T)$ have the same degree $n$, hence they are equal. But this implies that the $L$-series themselves are equal \cite[Lemma~2]{perlis}. Now \cite[Theorem 10.1]{CdSLMS} guarantees that $K \simeq K'$.
\end{proof}

The rest of the section is devoted to proving Theorem~\ref{theorem:dividing_L_series_implies_equal_degree}. Note that $n \geq n'$ as the degrees of $L_p(\widetilde{\chi}, T)$ and $L_p(\chi', T)$ are $n$ and $n'$ respectively.

Denote $\rho = \mathrm{Ind}^{G_\Qz}_{G_K} (\widetilde{\chi})$ and $\rho' = \mathrm{Ind}^{G_\Qz}_{G_{K'}} (\chi')$, so that $L(\widetilde{\chi}, s) = L(\rho, s)$ and $L(\chi',s) = L(\rho',s)$ by Artin induction \cite[VII.10.4.(iv)]{neukirch2013algebraic}. As a result $L_p(\rho', T)$ divides $L_p(\rho, T)$ for all prime numbers $p$. By construction $\rho$ factors through $\Gal(M/\Qz) \simeq C^n \rtimes G$. We show that this is the case for $\rho'$ as well.

\begin{lemma}
The induced representation $\rho'$ factors through $\Gal(M/\Qz)$. 
\end{lemma}

\begin{proof}
Let $M'$ be a finite Galois extension of $\Qz$ such that $\rho'$ factors through the Galois group $\Gal(M'/\Qz)$. Consider the Galois group $\Gal(MM'/\Qz)$ and note that both $\rho$ and $\rho'$ factor through this Galois group. Let $\tau \in \Gal(MM'/M)$. By the Chebotarev density theorem there is a prime $p$ (unramified in $MM'/\Qz)$ such that $\Frob{p} = \tau$. As $\rho$ factors through $\Gal(M/\Qz)$, $L_p(\rho, T) = (1-T)^n$. As a result we have $L_p(\rho',T) = (1-T)^{n'}$ and it follows that $\rho'(\tau)$ is the identity matrix. Hence $\rho'$ factors through $\Gal(M/\Qz)$.
\end{proof}

\begin{lemma}
For any $t \in C^n \rtimes G \simeq \Gal(M/\Qz)$ the eigenvalues of $\rho'(t)$ are a subset (with multiplicity) of those of $\rho(t)$. 
\end{lemma}

\begin{proof}
Let $t \in C^n \rtimes G$ and let $p$ be a rational prime such that $\Frob{p} = t$. Then $L_p(\rho, T) = \det(I_n - \rho(t)T)$ and $L_p(\rho, T) = \det(I_{n'} - \rho'(t)T)$, whose zeroes are the inverses of the eigenvalues of $\rho(t)$ and $\rho'(t)$ respectively. As $L_p(\rho',T)$ divides $L_p(\rho, T)$, the result follows.
\end{proof}

Let $\alpha_i = (1, \dots, 1, \zeta, 1, \dots, 1, e) \in C^n \rtimes G$, with $\zeta$ as the $i^\text{th}$ coordinate.

\begin{lemma}
The matrix $\rho(\alpha_i)$ is a diagonal matrix with value $\zeta$ at the $i^\text{th}$ row and value $1$ otherwise.
\end{lemma}

\begin{proof}
Note that in $C^n \rtimes G$, 
\[
\alpha_i (1, \dots, 1, g_j) = (1, \dots, 1, g_j) \alpha_{k},
\]
where $k$ is such that $g_j^{-1} g_i \in g_k H$. By definition of the induced representation, $\rho(\alpha_i)$ is a diagonal matrix with value $\widetilde{\chi}(\alpha_k)$ on the $j^\text{th}$ row. This value is $\zeta$ exactly when $k = 1$ and $1$ otherwise. As $g_j^{-1} g_i \in H$ if and only if $j = i$, the result follows. 
\end{proof}

\begin{lemma}\label{lemma:rho'_diagonal}
The matrix $\rho'(\alpha_i)$ is diagonal for any $1 \leq i \leq n$. 
\end{lemma}

\begin{proof}
The eigenvalues of $\rho(\alpha_i)$ are $\zeta$ and $n-1$ times $1$, hence there are two possibilities for the eigenvalues of $\rho'(\alpha_i)$:
\begin{itemize}
\item $\rho'(\alpha_i)$ has $n'$ eigenvalues $1$: then $\rho'(\alpha_i)$ is the identity matrix (which is certainly diagonal);
\item $\rho'(\alpha_i)$ has $n'-1$ eigenvalues $1$ and one eigenvalue $\zeta$: the trace of $\rho'(\alpha_i)$ then equals $n' - 1 + \zeta$, which also equals the sum of the diagonal entries. By \cite[Lemma 6.6]{CdSLMS} we find that the diagonal entries must all equal $1$ except a single one which equals $\zeta$. In particular $\rho'(\alpha_i)$ is diagonal. \qedhere
\end{itemize}
\end{proof}

\begin{lemma}\label{lemma:rho'_one_zeta}
For all $1 \leq i \leq n$ the matrix $\rho'(\alpha_i)$ has $n'-1$ diagonal entries equal to $1$ and one entry equal to $\zeta$.
\end{lemma}

\begin{proof}
From the proof of Lemma~\ref{lemma:rho'_diagonal} we have that $\rho'(\alpha_i)$ has at least $n'-1$ eigenvalues $1$, and that the last eigenvalue is either $\zeta$ or $1$. Suppose that for all $1 \leq i \leq n$, the matrix $\rho'(\alpha_i)$ has $n'$ eigenvalues $1$, i.e. $\rho'(\alpha_i)$ is the identity matrix. Let $t = \prod\limits_{1 \leq i \leq n} \alpha_i$. Then
\[
\rho'(t) = \prod_{1 \leq i \leq n} \rho'(\alpha_i)  = I_{n'}.
\]
Note, however, that $\rho(t) = \prod_{1 \leq i \leq n} \rho(\alpha_i) = \zeta I_{n}$. This contradicts the fact that the all eigenvalues of $\rho'(t)$ are eigenvalues of $\rho(t)$. 

Hence there is an $1 \leq i \leq n$ such that  $\rho'(\alpha_i)$ has $n'-1$ diagonal entries equal to $1$ and one entry equal to $\zeta$. However, as $(1, \dots, 1, g_i) \alpha_1 (1, \dots, 1, g_i^{-1}) =  \alpha_i$, we see that all $\alpha_j$ are conjugate, hence $\rho'(\alpha_j)$ has trace $n - 1 + \zeta$ for every $1\leq j\leq n$. It is therefore of the desired form.
\end{proof}

\begin{lemma}\label{lemma:rho'_zeta_different_location}
For any $1 \leq i < j \leq n$ the location of the $\zeta$ in $\rho'(\alpha_i)$ and $\rho'(\alpha_j)$ is different. 
\end{lemma}

\begin{proof}
Suppose it is the same for some $i,j$. Then $\rho'(\alpha_i \alpha_j)$ has an eigenvalue $\zeta^2$, which does not occur as an eigenvalue of $\rho(\alpha_i \alpha_j)$ (which has two eigenvalues $\zeta$ and $n-2$ eigenvalues $1$).
\end{proof}

\begin{proof}[Proof of Theorem~\ref{theorem:dividing_L_series_implies_equal_degree}]
The matrices $\rho'(\alpha_1), \dots, \rho'(\alpha_n)$ all have exactly one $\zeta$ on the diagonal (Lemma~\ref{lemma:rho'_one_zeta}) and the locations of these $\zeta$s are pairwise different (Lemma~\ref{lemma:rho'_zeta_different_location}), and therefore $n' \geq n$, hence $n' =  n$.
\end{proof}

\section*{Acknowledgements}
It is a pleasure to acknowledge the encouragement, suggestions, and remarks from Gunther Cornelissen. I would also like to express thanks to Hendrik Lenstra and Peter Stevenhagen for the proof of Lemma~\ref{lemma:prime_above_totally_split} in the case of ramified primes. Many thanks to the Max Planck Institute for Mathematics in Bonn for its financial support and inspiring atmosphere.

\end{document}